\documentclass[10pt, a4paper, twoside]{amsart}

\title{On Fomin--Kirillov Algebras for Complex Reflection Groups}
\date{\today}
\author{Robert Laugwitz}
\address{School of Mathematics,
University of East Anglia, 
Norwich Research Park, 
Norwich, NR4 7TJ}
\email{r.laugwitz@uea.ac.uk}
\urladdr{https://www.uea.ac.uk/mathematics/people/profile/r-laugwitz}

\usepackage{mathabx}
\usepackage{import}
\usepackage{amsmath}
\usepackage{amsthm}
\usepackage[alphabetic, initials]{amsrefs}
\usepackage[english]{babel}
\usepackage{url}
\usepackage{fancyhdr}
\usepackage{graphicx}
\usepackage[
colorlinks=true,
linkcolor=black, % einfache interne Verknpfungen
anchorcolor=black,% Ankertext
citecolor=black, % Verweise auf Literaturverzeichniseintrge im Text
urlcolor=black, % Farbe der URLs
]{hyperref}
\usepackage{geometry}
\usepackage[all]{xy}
\usepackage{accents}

\usepackage{tikz}
\usetikzlibrary{arrows,positioning}

\makeatletter
\def\imod#1{\allowbreak\mkern10mu({\operator@font mod}\,\,#1)}
\makeatother

\makeatletter
\newcommand{\superimpose}[2]{%
  {\ooalign{$#1\@firstoftwo#2$\cr\hfil$#1\@secondoftwo#2$\hfil\cr}}}
\makeatother

\newcommand{\op}[1]{\operatorname{#1}}

\newcommand{\ad}{\operatorname{ad}}

\newcommand{\Drin}{\operatorname{Drin}}

\newcommand{\ide}{\operatorname{Id}}

\providecommand{\fr}[1]{\mathfrak{#1}}

\providecommand{\op}[1]{\operatorname{#1}}

\newcommand{\mC}{\mathbb{C}}

\newcommand{\mN}{\mathbb{N}}

\newcommand{\cB}{\mathcal{B}}
\newcommand{\cE}{\mathcal{E}}

\newcommand{\cH}{\mathcal{H}}

\newcommand{\cS}{\mathcal{S}}

\newtheoremstyle{mystyle}% name of the style to be used
  {0.5cm}                   %Space above
  {0.5cm}                   %Space below
  {\normalfont}           %Body font
  {}                      %Indent amount (empty = no indent,
  {\itfont\bfseries}  %Thm head font
  {:}                     %Punctuation after thm head
  {0.3cm}              %Space after thm head: " " = normal interword
  {\thmname{#1}}
\newtheoremstyle{defstyle}% name of the style to be used
  {0.5cm}                   %Space above
  {0.5cm}                   %Space below
  {\normalfont}           %Body font
  {}     %Indent amount (empty = no indent,
  {\normalfont\bfseries}  %Thm head font
  {:}                     %Punctuation after thm head
  {0.3cm}              %Space after thm head: " " = normal interword
  {\thmname{#1}\thmnumber{ #2}\thmnote{ (#3)}}
\numberwithin{equation}{section}
\newtheorem{theorem}{Theorem}[section]

\newtheorem{corollary}[theorem]{Corollary}
\newtheorem{lemma}[theorem]{Lemma}
\newtheorem{conjecture}[theorem]{Conjecture}

\newtheorem{theorem*}{Theorem}

\theoremstyle{definition}
\newtheorem{definition}[theorem]{Definition}

\theoremstyle{remark}
\newtheorem{example}[theorem]{Example}
\newtheorem{remark}[theorem]{Remark}

\renewcommand{\sectionmark}[1]		% Schriftform fr \section = Kaptlchen
	{
	\markboth{\small\it \thesection{} #1}{}
	}


\begin{document}

\begin{abstract}
This note is an application of classification results for finite-dimensional Nichols algebras over groups. We apply these results to generalizations of Fomin--Kirillov algebras to complex reflection groups. First, we focus on the case of cyclic groups where the corresponding Nichols algebras are only finite-dimensional up to order four, and we include results about the existence of Weyl groupoids and finite-dimensional Nichols subalgebras for this class. Second, recent results by Heckenberger--Vendramin [\textit{ArXiv e-prints}, \texttt{1412.0857} (December 2014)] on the classification of Nichols algebras of semisimple group type can be used to find that these algebras are infinite-dimensional for many non-exceptional complex reflection groups in the Shephard--Todd classification.\end{abstract}
\maketitle

%\tableofcontents

%%%%%%%%%%%%%%%%%%%%%%%%%%%%%%%%%%%%%%%%%%%%%%%%%%%%%%%%%%%%%%%%%%%%%%%%%%%%%%%%%%%%%%%%%%%%%%%%%%%%%%%%%%

\section{Introduction}

The quadratic algebras $\cE_n$ were introduced by S. Fomin and A. N. Kirillov as an algebraic tool to study Schubert calculus and the cohomology of flag varieties in \cite{FK}. This class of algebras attracted attention in the study of pointed Hopf algebras as they are quadratic covers of certain Nichols algebras $\cB_n$ (and hence are braided Hopf algebras generated by primitive elements). It this framework, the Nichols algebras $\cB_n$ appeared independently in \cite[Section~5]{MS} in larger generality. The Hopf algebra structure on the bosonization of $\cE_n$ with the group algebra was also observed in \cite{FP}.
It was conjectured in \cite[3.1]{Maj2} that $\cE_n$ is infinite-dimensional for $n\geq 6$. This question, as well as the question whether $\cB_n$ is quadratic (and hence equals $\cE_n$), remain open (cf. \cite[Introduction]{AFGV}).

Fuelled by the fact that the positive parts $u_q(\mathfrak{n}^+)$ of the small quantum groups appear as examples, the question of classifying all finite-dimensional Nichols algebras has been a topic of active research (see surveys \cite{AS,AS2}). The classification of finite-dimensional Nichols algebras over abelian groups has been completed in \cite{Hec2} (in characteristic zero). While the classification of finite-dimension Nichols algebras of group type has been completed in the case of a semisimple Yetter--Drinfeld module in \cite{HLV,HV2}, it remains open in general (see also \cite{HLV}).

Generalizations $\cE_G$ of the Fomin--Kirillov algebras (and their corresponding Nichols algebras $\cB_G$, already defined in \cite{MS}), were considered in \cite{Baz} (these are related to the bracket algebras of \cite{KM}) and provide a generalization of the algebras $\cE_n$ (respectively, $\cB_n$), which constitute the $S_n$-case, to general Coxeter groups $G$. The Nichols algebras $\cB_G$ were studied more generally for complex reflection groups $G$ in \cite{BB}.
\begin{itemize}
\item In particular, it is possible to consider analogues of $\cB_G$ over the cyclic groups $C_n$, which can be viewed as complex reflection groups (of rank 1), where $C_n\leq \op{GL}(\mC)$ via a primitive $n$-th root of unity. As $C_n$ is abelian, this class of examples can be study using the machinery of classification of finite arithmetic root systems via the Weyl groupoid \cite{Hec1,Hec2}.
\item More generally, looking at the non-exceptional series of complex reflection groups $G(m,p,n)$ in the Shephard--Todd classification \cite{ST} we can use the results of \cite{HV2} to show that $\cB_n$ is infinite-dimensional for large classes of complex reflection groups (and hence $\dim \cE_n=\infty$ for many Coxeter groups).
\end{itemize}
Investigating these two points is the purpose of this note. 

In the cyclic group case treated in Section \ref{cyclicsection}, we derive that---similarly to the algebras $\cE_n$---the Nichols algebras $\cB(Y_{C_n})$ associated to cyclic groups by the construction of \cite{BB} also display rapid growth. In fact, they are infinite-dimensional for $n\geq 5$. We see that this class of Nichols algebras gives examples of different types depending on $n$: some are of Cartan type, others are not, and it is further shown that the associated Weyl groupoid is not of even well-defined in general for these examples, by showing that this e.g. fails for $n=6$ already. This means that no convenient PBW bases are available for these algebras in all cases, and the Gelfand--Kirillov dimension is not finite in general. In fact, only for $n=p$ a prime number do we obtain Nichols algebras of Cartan type. We conjecture that the Weyl groupoid exists if and only if $n$ is prime or equal to $4$. Computational evidence is used to verify this claim up to $n=1000$.

From the point of view of the theory of Nichols algebras, this means that we investigate examples of diagonal type, such that the braiding is given by scalars
\begin{align}
q_{ij}&=\xi^i, &\forall i,j\in I,
\end{align}
where $I$ is an indexing set for a diagonal basis with respect to the braiding and $\xi$ is a primitive $n$-th root of unity. The classification in \cite{Hec1} can be used to identify for which $n$ such braidings of rank at least two exist that have a finite root system, hence giving finite-dimensional Nichols algebras (which are Nichols subalgebras of $\cB_{C_n}$ coming from Yetter--Drinfeld submodules). We see that the maximal rank of such a Nichols algebra is three, and they only exist for $n$ divisible by one of the primes $p=2,3,5$ or 7 (Corollary \ref{subalgebras}, Table \ref{subsystems}). Note that the classification of pointed Hopf algebras via the lifting method also appear as exceptions \cite{AS2}.

For the case of a complex reflection group $G=G(m,p,n)$ we see in Section \ref{generalcase} that the rank of the associated Nichols algebra $\cB_G$ equals $m/p$ if $n\geq 3$ or $p$ is odd. Otherwise, the rank is $m/p+1$. One can see using results of \cite{HV2} that the associated Nichols algebra $\cB_G$ is infinite-dimensional whenever this rank is larger than one, except for the the groups $D_2=G(2,2,2)$ and $B_2=G(2,1,2)\cong G(4,4,2)$. However, in the $B_2$-case, $\cE_{B_2}$ is also infinite-dimensional according to \cite{KM}. Finite-dimensionality remains open in cases where the rank of $\cB_G$ equals one. Hence the only open cases left are among the groups $G(m,m,n)$ with $n\geq 3$ or $m$ odd. A conjecture in \cite{HLV} would imply that a full list of finite-dimensional elementary Fomin--Kirillov Nichols algebra from non-exceptional complex reflection groups (up to isomorphism) will be given by $\cB_G$ for $G=S_2,S_3,S_4,S_5$.

%%%%%%%%%%%%%%%%%%%%%%%%%%%%%%%%%%%%%%%%%%%%%%%%%%%%%%%%%%%%%%%%%%%%%%%%%%%%%%%%%%%%%%%%%%%%%%%%%%%%%%%%%%

\section{The Cyclic Group Case}\label{cyclicsection}

In the following, we work over a field $k$ of characteristic zero in order to apply \cite{Hec2}. However, the definition and some of the results hold more generally.

\subsection{Definitions}

We fix a generator $g$ for the cyclic group $C_n$, a primitive $n$-th root of unity $\xi$, and let $C_n\leq \op{GL}_1(\mC)$ via $\xi$. Hence $C_n$ is an irreducible complex reflection group with reflections $g,\ldots, g^{n-1}$. The Yetter--Drinfeld module $Y_{C_n}=\op{span}_\mC\langle s_1,\ldots,s_{n-1}\rangle$ is given by the $C_n$ action $g\cdot s_i=\xi s_i$ and the grading $\delta(s_i)=g^i\otimes s_i$. We denote the corresponding Nichols algebra of rank $n-1$ by
\[
\cB_{C_n}:=\cB(Y_{C_n}).
\]
This is a special case of the general definition of \cite[Section~7]{BB} for a complex reflection group (cf. \cite[Section~5]{MS}), and the investigation of $\cB_{C_n}$ will occupy the remainder of this section, while we will look at the more general case $G=G(m,p,n)$ in Section \ref{generalcase}. We refer to the algebras $\cB_{C_n}$ as \emph{Fomin--Kirillov Nichols algebras} (or \emph{FK Nichols algebras})\footnote{Note that the original Fomin--Kirillov algebras $\cE_n$ are only quadratic covers of their Nichols algebra quotients. It is unknown whether the Nichols algebras themselves are quadratic in the symmetric group case ($n\geq 6$). However, for cyclic groups most relations are not quadratic so we do not consider the quadratic cover.} for $C_n$. We observe that the braidings have the form
\begin{align}
\Psi(s_i\otimes s_j)=\xi^i s_j\otimes s_i.
\end{align}
Hence $\cB_{C_n}$ is of diagonal type (cf. \cite[(1-11)]{AS}), where $q_{ij}=\xi^i$. This implies that the generalized Dynkin diagram (cf. \cite[Definition~1]{Hec2}) of $\cB_{C_n}$ consists of vertices labelled by $\xi^i$ for $i=1,\ldots, n-1$ and edges
\begin{align}
\begin{tikzpicture}
\node [circle,draw,label=above:$\xi^i$] (1){};
\node [circle,draw,label=above:$\xi^j$] (2)[right of=1,node distance=1.5cm]{};
\draw  (1.east) -- (2.west) node [above,text centered,midway]
{$\xi^{i+j}$};
\end{tikzpicture}.
\end{align}
The underlying graph is almost complete, apart from there not being an edge between the vertices labelled $\xi^i$ and $\xi^{n-i}$ for all $i=1,\ldots, n-1$.

\begin{example}
We list the examples of small rank in Table \ref{smallrank}.
\begin{table}[b]
\begin{tabular}{ccccc}$
n$&gen. Dynkin diagram& gen. Cartan matrix& type& dimension\\\hline
2& $\vcenter{\hbox{\begin{tikzpicture}
\node [circle,draw,label=above:$-1$] (1){};
\end{tikzpicture}}\vspace*{4pt}}$&$\begin{pmatrix}
2\end{pmatrix}$
& $A_1,\text{Cartan}$&2\\
3&  $\vcenter{\hbox{\begin{tikzpicture}
\node [circle,draw,label=above:$\xi$] (1){};
\node [circle,draw,label=above:$\xi^{-1}$] (2)[right of=1,node distance=1.5cm]{};
\end{tikzpicture}}\vspace*{4pt}}$& $\begin{pmatrix}
2&0\\0&2
\end{pmatrix}$& $A_2\times A_2, \text{Cartan}$& 9\\
4&  $\vcenter{\hbox{\begin{tikzpicture}
\node [circle,draw,label=above:$\xi$] (1){};
\node [circle,draw,label=above:$-1$] (2)[right of=1,node distance=1.5cm]{};
\node [circle,draw,label=above:$\xi^{-1}$] (3)[right of=2,node distance=1.5cm]{};
\draw  (1.east) -- (2.west) node [above,text centered,midway]
{$\xi^{-1}$};
\draw  (2.east) -- (3.west) node [above,text centered,midway]
{$\xi$};
\end{tikzpicture}}\vspace*{4pt}}$
& $\begin{pmatrix}2&-1&0\\-1&2&-1\\0&-1&2\end{pmatrix}$& $A_3,\text{not Cartan, standard}$& 256\\
5&  $\vcenter{\hbox{\begin{tikzpicture}
\node [circle,draw,label=above left:$\xi$] (1){};
\node [circle,draw,label=above right:$\xi^2$] (2)[right of=1,node distance=1.5cm]{};
\node [circle,draw,label=below left:$\xi^{-2}$] (3)[below of=1,node distance=1.5cm]{};
\node [circle,draw,label=below right:$\xi^{-1}$] (4)[right of=3,node distance=1.5cm]{};
\draw  (1.east) -- (2.west) node [above,text centered,midway]
{$\xi^{-2}$};
\draw  (1.south) -- (3.north) node [left,text centered,midway]
{$\xi^{-1}$};
\draw  (3.east) -- (4.west) node [below,text centered,midway]
{$\xi^{2}$};
\draw  (2.south) -- (4.north) node [right,text centered,midway]
{$\xi$};
\end{tikzpicture}}\vspace*{4pt}}$&$\begin{pmatrix}
2&-2&-1&0\\-1&2&0&-2\\-2&0&2&-1\\0&-1&-2&2
\end{pmatrix}$&Cartan, not symmetric&$\infty$
\end{tabular}
\caption{Small rank examples: $\cB_{C_n}$, $n=2,3,4,5$}.
\label{smallrank}
\end{table}
We recognize the smallest rank cases as follows: $\cB_{C_2}$ is the positive part of the small quantum group of $\fr{sl}_2$ for $q=-1$, while $\cB_{C_3}$ is the braided tensor product  two copies of the positive part of the small quantum group $\fr{sl}_2$ for $q=e^{2/3\pi i}$. The case  $n=4$ appears in \cite[Table~2, Row~8]{Hec4} (cf. also \cite[Table~2, Row~8]{Hec2}) and is hence finite-dimensional. Using techniques of \cite[Theorem~2]{Kha}, we can show that a PBW basis for $\cB_{C_4}$ can be given by generators $s_{\alpha_i}$, labelled by six roots $\alpha_i$, and
\begin{align}
\cB_{C_4}=\op{span}_k\langle s_{\alpha_1}^{b_1}\ldots s_{\alpha_6}^{b_6}\mid 0\leq b_1,b_3\leq 3, 0\leq b_2,b_4,b_5, b_6 \leq 1\rangle,
\end{align}
where $\alpha_1$, $\alpha_2$, $\alpha_3$ are simple roots (corresponding to the generators $s_{\alpha_i}=s_i$ for $i=1,2,3$), and $\alpha_4=\alpha_1+\alpha_2$, $\alpha_5=\alpha_2+\alpha_3$, $\alpha_6=\alpha_1+\alpha_2+\alpha_3$.
\end{example}

\begin{lemma}
The braiding on the Nichols algebra $\cB_{C_n}$ is of Cartan type if and only if $n=p$ is prime.
\end{lemma}

\begin{proof}
A braiding is of Cartan type if $q_{ii}\neq 1$ for all $i$, and if $i\neq j$, then
\[
a_{ij}=-\op{min}\lbrace r\geq 0 \mid q_{ii}^rq_{ij}q_{ji}=1\rbrace
\]
exists and satisfies $0\leq -a_{ij}< \op{ord} q_{ii}$, i.e. $a_{ii}=2$ and $(a_{ij})$ is a generalized Cartan matrix (cf. \cite[1.2]{AS3}). In the series $\cB_{C_n}$ this means that we need $k$ such that
\begin{align}
\xi^{ik+j}&=1& \Longleftrightarrow&& ik&\equiv -j\mod n,
\end{align}
in which case we set $a_{ij}=-k+1$. This division modulo $n$ can always be solved if and only if $n$ is a prime $p$, in which case we choose $1\leq k\leq \op{ord}q_{ii}=p$.
\end{proof}

\begin{corollary}\label{fdthm}
The algebra $\cB_{C_n}$ is finite-dimensional if and only if $n\leq 4$.
\end{corollary}
\begin{proof}
We can see, using the classification of finite arithmetic root systems \cite{Hec2}, that for $n=2,3,4$ we obtain finite-dimensional Nichols algebras. The previous observations about the braiding imply that for $n\geq 5$ there always exists a cycle of length $n-1$ within the generalized Dynkin diagram of the braiding of $\cB_{C_n}$. Again using the classification in \emph{loc.cit.} we find that all such braidings give infinite-dimensional arithmetic root systems and hence infinite-dimensional Nichols algebras.
\end{proof}

\subsection{Finite-dimensional Nichols subalgebras}

The more general point of view to find all braidings of diagonal type which are of the form that
\begin{align}
q_{ij}&=\xi^i, &\forall i\in I,
\end{align}
where $I$ is a subset of $\lbrace 1,\ldots, n-1\rbrace$ and have a finite arithmetic root system is equivalent to finding the finite-dimensional Nichols subalgebras of $\cB_{C_n}$. That is, Nichols algebras that are braided (and graded) Hopf subalgebras of $\cB_{C_n}$ which arise from considering a restriction of the defining Yetter--Drinfeld module $Y_{C_n}$ to a subset of the diagonal basis. Again, a full answer to this question can be directly obtained from the classification of finite arithmetic root systems in \cite[Table~1--4]{Hec2}.

\begin{samepage}
\begin{corollary}$~$\label{subalgebras}
\begin{enumerate}
\item[(i)]
There are indecomposable finite-dimensional Nichols subalgebras in $\cB_{C_n}$ of rank 2 or higher if and only if $n$ is divisible by $p=3,5,7$ or $4$.
\item[(ii)]
$\cB_{C_4}$ is the only indecomposable Nichols subalgebras of the above form of rank 3.
\item[(iii)]
There are no finite-dimensional indecomposable Nichols subalgebras of the above form of rank 4 or higher.
\end{enumerate}
\end{corollary}
\end{samepage}

\begin{proof}
As a proof, consider the explicit list\footnote{In the table, $\xi$ is a primitive $n$-th root of unity. We omit reflections that invert given ones, or remain at one generalized Dynkin diagram.} in Table~\ref{subsystems} of the indecomposable finite arithmetic\footnote{For some computations e.g. of dimensions, we use \cite{SARNA}.} root systems of Nichols subalgebras of $\cB_{C_n}$ that appear in the classification in \cite{Hec2}. As $m$ dividing $n$ implies $\cB_{C_m}\leq \cB_{C_n}$, we only list those subsystems that are strictly smaller than the smallest $\cB_{C_n}$ in which they appear. The list only contains rank 2 systems (cf. \cite[Table~1]{Hec2} or \cite[Table~1]{Hec3}) as it is clear that each vertex gives a rank 1 subsystem for any $\cB_{C_n}$, and apart from $\cB_{C_4}$ itself no higher rank systems of this form exist.
\end{proof}

\begin{table}[b]
\begin{tabular}{cccccc}
$n$ & Weyl groupoid &Cartan matrices& Type & Dimension&\cite{Hec2}\\\hline
$4$& $\vcenter{\hbox{\begin{tikzpicture}
\node [circle,draw,label=above:$\xi$] (1){};
\node [circle,draw,label=above:$-1$] (2)[right of=1,node distance=1.5cm]{};
\draw  (1.east) -- (2.west) node [above,text centered,midway](a)
{$\xi^{-1}$};
\node [circle,draw,label=above:$-1$] (3)[below of=1,node distance=1.4cm]{};
\node [circle,draw,label=above:$-1$] (4)[right of=3,node distance=1.5cm]{};
\draw  (3.east) -- (4.west) node [above,text centered,midway](b)
{$\xi$};
\node [circle,draw,label=above:$\xi$] (6)[below right=1.4cm and 0.2cm of 3]{};
\node [circle,draw,label=above:$-1$] (5)[left of=6,node distance=1.5cm]{};
\draw  (5.east) -- (6.west) node [above,text centered,midway](c)
{$\xi^{-1}$};
\node [circle,draw,label=above:$-1$] (7)[below right=1.4cm and 0.8cm of 3]{};
\node [circle,draw,label=above:$\xi^{-1}$] (8)[right of=7,node distance=1.5cm]{};
\draw  (7.east) -- (8.west) node [above,text centered,midway](d)
{$\xi$};
\node [circle,draw,label=above:$-1$] (9)[below=2.8cm of 3]{};
\node [circle,draw,label=above:$-1$] (10)[right of=9,node distance=1.5cm]{};
\draw  (9.east) -- (10.west) node [above,text centered,midway](e)
{$\xi^{-1}$};
\node [circle,draw,label=above:$\xi^{-1}$] (11)[below=4.2cm of 3,node distance=1.2cm]{};
\node [circle,draw,label=above:$-1$] (12)[right of=11,node distance=1.5cm]{};
\draw  (11.east) -- (12.west) node [above,text centered,midway](f)
{$\xi$};
\path[every node/.style={font=\sffamily\small},->,>=stealth',auto,node distance=3cm, shorten <=5pt,
  thick,main node/.style={circle,draw,font=\sffamily\Large\bfseries}]
    (a) edge[bend left=0] node [left] {$s_2$} (b)
    (c) edge[bend left=-10] node [left] {$s_2$} (e)
    (d) edge[bend left=10] node [right] {$s_1$} (e)
    (e) edge[bend left=0] node [left] {$s_2$} (f);
\path[every node/.style={font=\sffamily\small},->,>=stealth',auto,node distance=3cm, shorten <=5pt,
  thick,main node/.style={circle,draw,font=\sffamily\Large\bfseries}]
    (b.south) edge[bend left=-10] node [left] {$s_1$} (c)
    (b.south) edge[bend left=10] node [right] {$s_2$} (d);
\end{tikzpicture}}\vspace*{4pt}}$& $\begin{pmatrix} 2&-1\\-1&2\end{pmatrix}$&\begin{tabular}{c}$A_2$, standard\\not Cartan\end{tabular}& $2^2\cdot 4=16$& \begin{tabular}{c}Table 1, \\ row 2\end{tabular}\\\hline
$5$& $\vcenter{\hbox{\begin{tikzpicture}
\node [circle,draw,label=above:$\xi$] (1){};
\node [circle,draw,label=above:$\xi^2$] (2)[right of=1,node distance=1.5cm]{};
\draw  (1.east) -- (2.west) node [above,text centered,midway]
{$\xi^{-2}$};
\end{tikzpicture} }}$& $\begin{pmatrix} 2&-2\\-1&2\end{pmatrix}$&\begin{tabular}{c}$B_2$, Cartan\end{tabular}& $5^4=625$& \begin{tabular}{c}Table 1, \\ row 3\end{tabular}\\[10pt]\hline
$6$& $\vcenter{\hbox{\begin{tikzpicture}
\node [circle,draw,label=above:$\xi$] (1){};
\node [circle,draw,label=above:$-1$] (2)[right of=1,node distance=1.5cm]{};
\draw  (1.east) -- (2.west) node [above,text centered,midway](a)
{$\xi^{-2}$};
\node [circle,draw,label=above:$\xi^2$] (3)[below of=1,node distance=1.4cm]{};
\node [circle,draw,label=above:$-1$] (4)[right of=3,node distance=1.5cm]{};
\draw  (3.east) -- (4.west) node [above,text centered,midway](b)
{$\xi^2$};
\path[every node/.style={font=\sffamily\small},->,>=stealth',auto,node distance=3cm, shorten <=5pt,
  thick,main node/.style={circle,draw,font=\sffamily\Large\bfseries}]
    (a) edge[bend left=0] node [left] {$s_2$} (b);
\end{tikzpicture}}\vspace*{4pt}}$& $\begin{pmatrix} 2&-2\\-1&2\end{pmatrix}$&\begin{tabular}{c}$B_2$, standard\\not Cartan\end{tabular}& $2^2\cdot 3\cdot 6=72$& \begin{tabular}{c}Table 1, \\ row 4\end{tabular}\\[10pt]
& $\vcenter{\hbox{\begin{tikzpicture}
\node [circle,draw,label=above:$\xi^{-2}$] (1){};
\node [circle,draw,label=above:$\xi$] (2)[right of=1,node distance=1.5cm]{};
\draw  (1.east) -- (2.west) node [above,text centered,midway](a)
{$\xi^{-1}$};
\node [circle,draw,label=above:$\xi^{-2}$] (3)[below of=1,node distance=1.4cm]{};
\node [circle,draw,label=above:$-1$] (4)[right of=3,node distance=1.5cm]{};
\draw  (3.east) -- (4.west) node [above,text centered,midway](b)
{$-1$};
\path[every node/.style={font=\sffamily\small},->,>=stealth',auto,node distance=3cm, shorten <=5pt,
  thick,main node/.style={circle,draw,font=\sffamily\Large\bfseries}]
    (a) edge[bend left=0] node [left] {$s_1$} (b);
\end{tikzpicture}}\vspace*{4pt}}$& $\begin{pmatrix} 2&-2\\-1&2\end{pmatrix}$&\begin{tabular}{c}$B_2$, standard\\not Cartan\end{tabular}& $2\cdot 3^2\cdot 6=108$& \begin{tabular}{c}Table 1, \\ row 5\end{tabular}\\[10pt]
& $\vcenter{\hbox{\begin{tikzpicture}
\node [circle,draw,label=above:$\xi^{2}$] (1){};
\node [circle,draw,label=above:$-1$] (2)[right of=1,node distance=1.5cm]{};
\draw  (1.east) -- (2.west) node [above,text centered,midway]
{$\xi^{-1}$};
\node [circle,draw,label=above:$\xi^{-2}$] (3)[below of=1,node distance=1.4cm]{};
\node [circle,draw,label=above:$-1$] (4)[right of=3,node distance=1.5cm]{};
\draw  (3.east) -- (4.west) node [above,text centered,midway](b)
{$\xi$};
\path[every node/.style={font=\sffamily\small},->,>=stealth',auto,node distance=3cm, shorten <=5pt,
  thick,main node/.style={circle,draw,font=\sffamily\Large\bfseries}]
    (a) edge[bend left=0] node [left] {$s_2$} (b);
\end{tikzpicture}}\vspace*{4pt}}$& $\begin{pmatrix} 2&-2\\-1&2\end{pmatrix}$&\begin{tabular}{c}$B_2$, standard\\not Cartan\end{tabular}& $2^2\cdot 3^3=36$& \begin{tabular}{c}Table 1, \\ row 6\end{tabular}\\[10pt]\hline
$7$& $\vcenter{\hbox{\begin{tikzpicture}
\node [circle,draw,label=above:$\xi$] (1){};
\node [circle,draw,label=above:$\xi^{3}$] (2)[right of=1,node distance=1.5cm]{};
\draw  (1.east) -- (2.west) node [above,text centered,midway]
{$\xi^{-3}$};
\end{tikzpicture}}\vspace*{4pt}}$& $\begin{pmatrix} 2&-3\\-1&2\end{pmatrix}$&\begin{tabular}{c}$G_2$, Cartan\end{tabular}& $7^6=117649$& \begin{tabular}{c}Table 1, \\ row 10\end{tabular}\\[10pt]\hline
$8$& $\vcenter{\hbox{\begin{tikzpicture}
\node [circle,draw,label=above:$\xi^{2}$] (1){};
\node [circle,draw,label=above:$\xi^{-1}$] (2)[right of=1,node distance=1.5cm]{};
\draw  (1.east) -- (2.west) node [above,text centered,midway](a)
{$\xi$};
\node [circle,draw,label=above:$\xi^{2}$] (3)[below of=1,node distance=1.4cm]{};
\node [circle,draw,label=above:$-1$] (4)[right of=3,node distance=1.5cm]{};
\draw  (3.east) -- (4.west) node [above,text centered,midway](b)
{$-\xi^{-1}$};
\node [circle,draw,label=above:$\xi$] (5)[below of=3,node distance=1.4cm]{};
\node [circle,draw,label=above:$-1$] (6)[right of=5,node distance=1.5cm]{};
\draw  (5.east) -- (6.west) node [above,text centered,midway](c)
{$-\xi$};
\path[every node/.style={font=\sffamily\small},->,>=stealth',auto,node distance=3cm, shorten <=5pt,
  thick,main node/.style={circle,draw,font=\sffamily\Large\bfseries}]
    (a) edge[bend left=0] node [left] {$s_1$} (b)
    (b) edge[bend left=0] node [left] {$s_2$} (c);
\end{tikzpicture}}\vspace*{4pt}}$& $\begin{pmatrix} 2&-3\\-1&2\end{pmatrix}$&\begin{tabular}{c}$G_2$, standard\\not Cartan\end{tabular}& $2\cdot 4^2\cdot 8=256$& \begin{tabular}{c}Table 1, \\ row 11\end{tabular}\\[10pt]\hline
$10$& $\vcenter{\hbox{\begin{tikzpicture}
\node [circle,draw,label=above:$\xi^{-3}$] (1){};
\node [circle,draw,label=above:$-1$] (2)[right of=1,node distance=1.5cm]{};
\draw  (1.east) -- (2.west) node [above,text centered,midway](a)
{$\xi^2$};
\node [circle,draw,label=above:$\xi^{4}$] (3)[below of=1,node distance=1.4cm]{};
\node [circle,draw,label=above:$-1$] (4)[right of=3,node distance=1.5cm]{};
\draw  (3.east) -- (4.west) node [above,text centered,midway](b)
{$\xi^{-2}$};
\path[every node/.style={font=\sffamily\small},->,>=stealth',auto,node distance=3cm, shorten <=5pt,
  thick,main node/.style={circle,draw,font=\sffamily\Large\bfseries}]
    (a) edge[bend left=0] node [left] {$s_2$} (b);
\end{tikzpicture}}\vspace*{4pt}}$& \begin{tabular}{c}$\begin{pmatrix} 2&-4\\-1&2\end{pmatrix}$ \\[10pt] $\begin{pmatrix} 2&-3\\-1&2\end{pmatrix}$ \end{tabular}&\begin{tabular}{c}\emph{not} standard\end{tabular}& \begin{tabular}{c}$2^4\cdot 5^2\cdot 10^2$\\ $=40000$\end{tabular}& \begin{tabular}{c}Table 1, \\ row 13\end{tabular}
\end{tabular}
\caption{Finite subsystems for $\cB_{C_n}$}
\label{subsystems}
\end{table}

We note that the non-existence of finite-dimensional Nichols algebras with connected generalized Dynkin diagram of rank strictly larger than three over cyclic groups was shown in \cite{WZZ} using results of \cite{Hec2}.

\subsection{Existence of Weyl groupoids}

Note that the construction of the Weyl groupoid as defined in \cite[Section~5]{Hec1}, see also \cite[Section~2]{Hec2}, is valid for a large class, but not all diagonal braidings. It fails if, after reflecting at a vertex, we find that the new scalar $q'_{ii}=1$. We now investigate the question for which $n$ the Weyl groupoid can be defined for the diagonal type braiding of $\cB_{C_n}$.

\begin{lemma}\label{counterex}
The Weyl groupoid associated to the braiding on $\cB_{C_n}$ is \emph{not} defined at least in the cases where $pr$ divides $n$, and $p$ is a prime dividing $2r-1$.
\end{lemma}
\begin{proof}
First, consider the case $n=6=2\cdot 3$ separately. In this case, consider e.g. the sub-diagram of the generalized Dynkin diagram and Cartan matrix given by 
\begin{align}\begin{array}{cc}
\vcenter{\hbox{\begin{tikzpicture}
\node [circle,draw,label=above:$-\xi$] (1){};
\node [circle,draw,label=above:$\xi^{-1}$] (2)[right of=1,node distance=1.5cm]{};
\draw  (1.east) -- (2.west) node [above,text centered,midway]
{$-1$};
\end{tikzpicture}}}&\qquad\begin{pmatrix}
2&-2\\-3&2
\end{pmatrix},
\end{array}
\end{align}
where $\xi^3=-1$. The first vertex $-\xi$ is not Cartan and reflecting at it, using the generalized Cartan matrix, gives
\begin{align}\begin{array}{cc}
\vcenter{\hbox{\begin{tikzpicture}
\node [circle,draw,label=above:$-\xi$] (1){};
\node [circle,draw,label=above:$-1$] (2)[right of=1,node distance=1.5cm]{};
\draw  (1.east) -- (2.west) node [above,text centered,midway]
{$\xi^{-1}$};
\end{tikzpicture}}}&\qquad\begin{pmatrix}
2&-2\\-1&2
\end{pmatrix}.
\end{array}
\end{align}
Here the second vertex is not Cartan, and reflecting at it gives $\xi^0=1$ at the first vertex.

Now let $n=pr$ where $p>2$ is a prime and $p$ does not divide $r>2$. In this case, we can consider the sub-diagram of the form
\begin{align}\begin{array}{cc}
\vcenter{\hbox{\begin{tikzpicture}
\node [circle,draw,label=above:$\xi^p$] (1){};
\node [circle,draw,label=above:$\xi^r$] (2)[right of=1,node distance=1.5cm]{};
\node [circle,draw,label=above:$\xi^{-p}$] (3)[right of=2,node distance=1.5cm]{};
\draw  (1.east) -- (2.west) node [above,text centered,midway]
{$\xi^{p+r}$};
\draw  (2.east) -- (3.west) node [above,text centered,midway]
{$\xi^{-p+r}$};
\end{tikzpicture}}}&\qquad\begin{pmatrix}
2&-(r-1)&0\\a_{21}&2&a_{23}\\0&-(r-1)&2
\end{pmatrix}.
\end{array}
\end{align}
The first and third vertex are \emph{not} Cartan as $p$ does not divide $r$. Note that iterated reflection at these vertices (denoted by $s_1$ and $s_3$) will leave the first and third, as well as the entries in the first and third row of the Cartan matrix unchanged. Further observe that the reflection $s_1$ gives the label $\xi^{r^2}$ at the second vertex, $s_3s_1$ gives $\xi^{(2r-1)r}$. However, if $p$ divides $2r-1$, then this vertex has label $1$, i.e. the reflection does not exist. Finally, note that if $p$ divides $2r-1$, it never divides $r$ as the numbers are coprime.
\end{proof}

In particular, we see that $\cB_{C_n}$ does not have a Weyl groupoid for $n=6$, which is the smallest such value. This also fails for multiples of $n=6,15,28,33,40,51,65,77,91$, etc. We expect this to fail for other non-prime $n$ too as Lemma \ref{counterex} only uses a small part of the reflections for $\cB_{C_{pr}}$. As an example, we provide the full Weyl groupoid for the exceptional case $4$.

\begin{example}
The case $n=4$ is of standard type, i.e. the Cartan matrix is the same at each object of the groupoid (namely, of type $A_3$). The full Weyl groupoid is displayed in Figure \ref{groupoid4}.
\begin{figure}[h]
\begin{tikzpicture}[->,>=stealth',auto,node distance=3cm,
  thick,main node/.style={circle,draw,font=\sffamily\Large\bfseries}]

\node [draw=none,fill=none] (x){};  
\node [draw=none,fill=none] (1)[left of=x,node distance=2cm]{$a_2$};
\node [draw=none,fill=none] (3)[left of=1,node distance=2cm]{$a_1$};
\node [draw=none,fill=none] (2)[below of=x,node distance=2cm]{$a_3$};
\node [draw=none,fill=none] (4)[above of=x,node distance=2cm]{$b_3$};
\node [draw=none,fill=none] (5)[right of=x,node distance=2cm]{$b_2$};
\node [draw=none,fill=none] (6)[right of=5,node distance=2cm]{$b_1$};

  \path[every node/.style={font=\sffamily\small}]
    (1) edge[bend left=10] node [below] {$s_2$} (3)
    (3) edge[bend left=10] node [above] {$s_2$} (1)
    (1) edge[bend left=-20] node [above] {$s_1$} (2)
    (2) edge[bend left=40] node [below] {$s_1$} (1)
    (1) edge[bend left=40] node [left] {$s_3$} (4)
    (4) edge[bend left=-20] node [right] {$s_3$} (1)
    (2) edge[bend left=-20] node [left] {$s_3$} (5)
    (5) edge[bend left=40] node [right] {$s_3$} (2)
    (4) edge[bend left=40] node [above] {$s_1$} (5)
    (5) edge[bend left=-20] node [below] {$s_1$} (4)
    (5) edge[bend left=10] node [above] {$s_2$} (6)
    (6) edge[bend left=10] node [below] {$s_2$} (5);
\end{tikzpicture}
\caption{Weyl groupoid of $\cB_{C_4}$}
\label{groupoid4}
\end{figure}

Here, all endomorphisms are the identity. If an arrow $s_1,s_2,s_3$ is not drawn from some object, then it is the identity on that object. The generalized Dynkin diagrams at the objects are:
\begin{align}
a_1\qquad\qquad&\longleftrightarrow \qquad\qquad\vcenter{\hbox{\begin{tikzpicture}
\node [circle,draw,label=above:$\xi$] (1){};
\node [circle,draw,label=above:$-1$] (2)[right of=1,node distance=1.5cm]{};
\node [circle,draw,label=above:$\xi^{-1}$] (3)[right of=2,node distance=1.5cm]{};
\draw  (1.east) -- (2.west) node [above,text centered,midway]
{$\xi^{-1}$};
\draw  (2.east) -- (3.west) node [above,text centered,midway]
{$\xi$};
\end{tikzpicture}}}\\
a_2\qquad\qquad&\longleftrightarrow \qquad\qquad\vcenter{\hbox{\begin{tikzpicture}
\node [circle,draw,label=above:$-1$] (1){};
\node [circle,draw,label=above:$-1$] (2)[right of=1,node distance=1.5cm]{};
\node [circle,draw,label=above:$-1$] (3)[right of=2,node distance=1.5cm]{};
\draw  (1.east) -- (2.west) node [above,text centered,midway]
{$\xi$};
\draw  (2.east) -- (3.west) node [above,text centered,midway]
{$\xi^{-1}$};
\end{tikzpicture}}}\\
a_3\qquad\qquad&\longleftrightarrow \qquad\qquad\vcenter{\hbox{\begin{tikzpicture}
\node [circle,draw,label=above:$-1$] (1){};
\node [circle,draw,label=above:$\xi$] (2)[right of=1,node distance=1.5cm]{};
\node [circle,draw,label=above:$-1$] (3)[right of=2,node distance=1.5cm]{};
\draw  (1.east) -- (2.west) node [above,text centered,midway]
{$\xi^{-1}$};
\draw  (2.east) -- (3.west) node [above,text centered,midway]
{$\xi^{-1}$};
\end{tikzpicture}}}
\end{align}
Further, $b_i$ corresponds to $a_i$ where $\xi$ is interchanged with $\xi^{-1}$.
\end{example}

\subsection{Computation results}
We can use a MATLAB\textsuperscript{\textregistered} script to verify if there are any other $n$ for which the Weyl groupoid does not exists. This is possible, as computing the reflection in the Weyl groupoid (cf. \cite[(13)]{Hec1})  for the braiding of $\cB_{C_n}$ reduces to $\op{mod} n$ computations. We checked that for $n\leq 1000$ the Weyl groupoid is only defined for $n$ a prime number or $n=4$. This leads to the following conjecture:

\begin{conjecture}\label{conjecture}
The Weyl groupoid associated to the braiding of $\cB_{C_n}$ is only defined for $n$ a prime number or $n=4$.
\end{conjecture}

The computer check is based on the following method: We do not need to check the statement for $n=kr$, if the Weyl groupoid does not exist for $r$. We can also rule out all $n=pr$ for $p$ prime covered by Lemma \ref{counterex} as well as cases where $n$ is prime (in these cases, the groupoid always exists as it is of Cartan type). This implies that the computer will only check the cases $n=pq$ prime pairs $p,q$ from some point onwards (unless the conjecture is false). We found that for most $n$ it is quickest to compute the reflections
$s_js_is_p$ where $p$ is the smallest prime dividing $n$ and $i\neq j$ start from $1,2$ and increase until some reflection fails to exist. This can be identified as a situation in which $q'_{ii}=1$ after applying $s_js_is_p$. Computational complexity is often larger if $n=pq$ for two distinct primes $q\gg p$ and $p>2$. 

Conjecture \ref{conjecture} would imply that $\cB_{C_n}$ has infinite Gelfand--Kirillov dimension for $n\geq 5$ not prime. This follows from the observations in \cite[Section~3]{Hec1}. In fact, finite Gelfand--Kirillov dimension implies that for $i\neq j$ the number
\begin{align}
-a_{ij}=\op{min}\lbrace m \in \mN_0\mid
(m + 1)_{q_{ii}}(q_{ii}^mq_{ij}q_{ji}-1) = 0\rbrace
\end{align}
is well-defined. If $q=1$, and $q_{ij}q_{ji}\neq 1$, then this is not the case. In our computations for $n>4$ not prime, we obtain that $q_{ii}=1$ for some $i$ after reflection, but there always exists $j\neq i$ with $q_{ij}q_{ji}\neq 0$ as the connected components of the generalized Dynkin diagram are preserved under reflection.\footnote{It is a not known that finite Gelfand--Kirillov-dimension implies finiteness of the root system, and hence of the Nichols algebra (cf. \cite[Section~3]{Hec1}). The Gelfand--Kirillov dimension for $\cB_{C_p}$ for primes $p>3$ might also be infinite.}

%%Say something about PBW basis in prime case? Too hard!

%%%%%%%%%%%%%%%%%%%%%%%%%%%%%%%%%%%%%%%%%%%%%%%%%%%%%%%%%%%%%%%%%%%%%%%%%%

\section{The General Case \texorpdfstring{$G=G(m,p,n)$}{G=G(m,n,p)}}\label{generalcase}

We conclude this note by applying classification results for finite-dimensional Nichols algebras \cite{HV2} of semisimple group type to Fomin--Kirillov algebras associated to more general non-exceptional complex reflection groups, which we study here more explicitly. In the rank one case of elementary Nichols algebras, which is not covered by these methods, we include a conjecture derived from \cite{HLV}.

\subsection{Definition and Description}

We consider complex reflection groups of the series $G(m,p,n)$ appearing in the classification of \cite{ST} and investigate their associated Nichols algebras generalizing the Fomin--Kirillov Nichols algebras. There are also 34 exceptional groups, which we do not consider here.

\begin{definition}
Let $m,n$ be integers and $p$ a divisor of $m$. The group $G(m,p,n)$ is defined as all transformations of $\mC^n$ of the form
\begin{align}
x_i&\mapsto \theta^{\nu_i} x_{\sigma(i)}, &\sigma \in S_n,
\end{align}
where $\theta$ is a fixed primitive $n$-th root of unity and $\sum_i \nu_i=0 \mod p$.
We denote this group element by $\theta^{\nu}\sigma$. The group $G(m,p,n)$ naturally acts on $\mC^n$ by definition and contains $S_n$ as a subgroup.
\end{definition}

\begin{example}
We can recognize some of the classical series of Coxeter groups in $G(m,p,n)$:
\begin{itemize}
\item $G(1,1,n)=S_n$
\item $G(2,1,n)$ has $\theta=-1$ and any sign change given by powers $\nu=(\nu_1,\ldots,\nu_n)\in C_2^n$ appears in the group. Hence this groups is the Weyl group $B_n\cong C_2^n\rtimes S_n$.
\item In $G(2,2,n)$ the sum $\sum_i \mu_i=0 \mod 2$. Hence this group is $D_n\cong C_2^{n-1}\rtimes S_n$.
\item The group $G(m,m,2)$ gives the Coxeter group $I_2(m)$, the $n$th dihedral group $Dih_n$ of order $2n$. As a special case, $G_2$ appears for $m=6$.
\end{itemize}
\end{example}

A full list of reflections $\cS_G$ of $G$ can be given as follows: First, let $s=\theta^\nu 1$. In this case,
\begin{align}
(1-s)x_i&=(1-\theta^{\nu_i})x_i,& \forall i,
\end{align} 
which has a one-dimensional image if and only if $\nu_i=0$ for all but precisely one $i$, denote $k:=\nu_i$. Then, however, we require $k=0 \mod p$. Hence $p$ divides $k$. Consider the order $d$ of $k$ modulo $m$. This is a divisor of $m/p$. For each $d$, there are $n \varphi(d)$ reflections, and their order is $d$. We denote such a reflection by $s_i^k$, it maps $x_i\mapsto \theta^k x_i$. Second, if $\sigma=(ij)$ is a transposition, then $\theta^{\nu}(ij)$ is a reflection if and only if $\nu_i=-\nu_j \mod m$, and $\nu_k=0 \mod m$ for all $k\neq i,j$.
We denote this reflection of order two by $\theta^k(ij)$, for $k=\nu_i$. These two types describe all complex reflections in $G$.

Now, for a given reflection $s\in \cS_G$ we specify a choice of roots $\alpha_s^*$ and coroots $\alpha_s$ satisfying
\begin{align}\label{signreflections}
sx_i&=x_i-(\alpha^*_s,x_i)\alpha_s,& \forall i=1,\ldots, n.
\end{align}
For this, we denote generators of the dual space of $\mC^n$ by $y_1,\ldots, y_n$. First, for $s_i^k$, we choose 
\begin{align}
\alpha_{s^k_i}^*&=y_i, &\alpha_{s_i^k}&:=(1-\theta^k)x_i.
\end{align}
Second, for the reflection $\theta^k(ij)$, we can choose
\begin{align}
\alpha_{\theta^k(ij)}^*&=y_i-\theta^{-k}y_j,x_i,&\alpha_{\theta^k(ij)}&=x_i-\theta^k x_j.
\end{align} 

To study Fomin--Kirillov algebras associated to $G=G(m,p,n)$, we need to understand the conjugation action of reflections on reflections (note that $(\theta^k(ij))^2=1$).
\begin{align}
s_i^ks_j^k(s_i^k)^{-1}&=s_j^k,& \forall i,j,k,l\label{conj1}\\
s_t^l\theta^k(ij)(s_t^l)^{-1}&=\theta^k(i,j), &t\neq i,j,\label{conj2}\\
s_i^l\theta^k(ij)(s_i^l)^{-1}&=\theta^{k-l}(ij),\label{conj3}\\
s_j^l\theta^k(ij)(s_j^l)^{-1}&=\theta^{k+l}(ij),\label{conj4}\\
\theta^k(ij)\theta^l(ij)\theta^k(ij)&=\theta^{2k-l}(ij),&\forall k,l,\label{conj5}\\
\theta^k(ij)\theta^l(ab)\theta^k(ij)&=\theta^{l}(ab), & \lbrace i,j\rbrace\cap \lbrace a,b\rbrace=\emptyset, \label{conj6}\\
\theta^k(ij)\theta^l(it)\theta^k(ij)&=\theta^{-k+l}(jt),&j<t,\label{conj7}\\
\theta^k(ij)\theta^l(it)\theta^k(ij)&=\theta^{k-l}(tj),&j>t,\label{conj8}\\
\theta^k(ij)\theta^l(ti)\theta^k(ij)&=\theta^{k+l}(tj),\label{conj9}\\
\theta^k(ij)\theta^l(tj)\theta^k(ij)&=\theta^{k-l}(it),&i<t,\label{conj10}\\
\theta^k(ij)\theta^l(tj)\theta^k(ij)&=\theta^{-k+l}(ti),&i>t,\label{conj11}\\
\theta^k(ij)\theta^l(jt)\theta^k(ij)&=\theta^{k+l}(it),\label{conj12}\\
\theta^k(ij)s_t^l\theta^k(ij)&=s_t^l, &t\neq i,j,\label{conj13}\\
\theta^k(ij)s_i^l\theta^k(ij)&=s_j^l,\label{conj14}\\
\theta^k(ij)s_j^l\theta^k(ij)&=s_i^l.\label{conj15}
\end{align}

Following \cite[7.8--7.10]{BB} can make a convenient choice of $\Drin(kG)$-character $\lambda$ in order to define the Fomin--Kirillov algebras. This character can be chosen to take values  $\pm 1$ for Coxeter group. Given a choice of roots and coroots for $G$, the character $\lambda$ is determined by $g\triangleright \alpha_s^*=\lambda(g,s)\alpha_{gsg^{-1}}^*$. We compute its values for reflections (which is sufficient):
\begin{align}
\lambda(s_i^l,s_j^k)&=\begin{cases}\theta^{-l} &i=j \\
1 & i\neq j,\end{cases}
&\lambda(s_t^l, \theta^k(ij))&=\begin{cases}\theta^{-l} ,&t=i\\ 1 ,&i\neq j\\\end{cases}\\
\lambda(\theta^k(ij),s_t)&=\begin{cases}1,&t\neq i,j,\\ \theta^{-k},& t=i,\\ \theta^k,& t=j.\end{cases}
\\
\lambda(\theta^k(ij),\theta^l(it))&=\begin{cases}\theta^{-k},&j<t\\ -\theta^{-l},&j>t,\end{cases},
&\lambda(\theta^k(ij),\theta^l(ti))&=1,\\
\lambda(\theta^k(ij),\theta^l(tj))&=\begin{cases}-\theta^{k-l},&i<t\\ 1,&i>t,\end{cases},
&\lambda(\theta^k(ij),\theta^l(jt))&=\theta^k,\\
\lambda(\theta^k(ij),\theta^l(ij))&=-\theta^{k-l},
&\lambda(\theta^k(ij),\theta^l(ab))&=1,
\end{align}
where $\lbrace i,j\rbrace\cap \lbrace a,b\rbrace=\emptyset$.

\begin{definition}[\cite{BB}]Let $G$ be a complex reflection group with reflections $\cS_G$. 
\begin{enumerate}
\item[(i)] Define the Yetter--Drinfeld module $Y_G$ to be $\op{span}_k\langle r_s\mid s\in \cS\rangle$, with coaction $\delta (r_s)=s\otimes r_s$ and action $g\triangleright r_s=\lambda(g,s)r_{gsg^{-1}}$. It comes with the braiding $\Psi_G(r_s\otimes r_t)=(s\triangleright r_t)\otimes t_s$.
\item[(ii)]
We refer to the  Nichols algebra $\cB_G:=\cB(Y_G)$ of $Y_G$ as the \emph{Fomin--Kirillov Nichols algebra} (or \emph{FK Nichols algebra}) of $G$.
\item[(iii)] 
Let $G$ be a Coxeter group. Then the quadratic algebra $\cE_G:=T(Y_G)/(\ker \Psi_G +\ide_{Y_G\otimes Y_G})$ associated to the Yetter--Drinfeld module $Y_G$ is called the \emph{Fomin--Kirillov algebra} (or \emph{FK algebra}) of $G$.
\end{enumerate}
\end{definition}

Note that the FK Nichols algebra of the cyclic group with $m$ elements considered in Section \ref{cyclicsection} arises as $\cB_G$ in the case $G=G(mp,p,1)$, while the algebras $\cE_n=\cE_{S_n}$ are the original algebras studied in \cite{FK}.

\subsection{Semisimple Decomposition}

In this section, we consider complex reflection groups $G=G(m,p,n)$ where $n\geq 2$, i.e. \emph{not} cyclic groups and work over a field $k$ of characteristic zero.

\begin{lemma}\label{ranklemma} The FK Nichols algebras $\cB_G$ for the complex reflection groups $G=G(m,p,n)$ have the following ranks, with corresponding dimensions of the simple YD modules $Y_G$. Note that
\begin{align}
\dim Y_G&=n\left(m\frac{(n-1)}{2}+\frac{m}{p}-1\right).
\end{align} 
\begin{enumerate}
\item[(i)]  If $n\geq 3$, the rank of $\cB_G$ is $m/p$. There is one YD summand of dimension $mn(n-1)/2$ spanned by $r_s$ for the order two reflections $s=\theta^k(ij)$. For each $d|(m/p)$, $d>1$, we have $\varphi(n)$ $n$-dimensional YD summands generated by the reflections $\lbrace r_s\mid s=s^k_i, i=1,\ldots,m\rbrace$ for $k$ of order $d$ modulo $m$.
\item[(ii)] If $n=2$, and $p$ is odd, then the same as in (i) applies.
\item[(iii)] If $n=2$ and $p=2b$ is even, then $\cB_G$ for $G=G(2a,2b,2)$ has rank $a/b+1$. The YD summand of order two reflections splits as a direct sum of two submodules $V^{\op{odd}}\oplus V^{\op{even}}$, $V^{\op{even}}$ is spanned by $\lbrace r_s\mid s=\theta^{2k}(12), k=0,\ldots,a-1\rbrace$, and $V^{\op{odd}}$ is spanned by $\lbrace r_s\mid s=\theta^{2k+1}(12), k=0,\ldots,a-1\rbrace$.
\end{enumerate}
\end{lemma}
\begin{proof}
This follows from considering the relations (\ref{conj1}--\ref{conj15}) and distinguishing cases. Note that the action in the YD modules $Y_G$ is given by conjugation up to scalars (given by $\lambda$). Hence, in order to find a direct sum decomposition into simples of $Y_G$, we need to determine the conjugacy classes of reflections in $G$. We write $r\sim s$ if the reflections $r$ and $s$ are conjugate. First, note that, if $p\neq m$, we have reflections of the form $s_i^k$, for any $k\neq 0 \mod m$ that is divisible by $p$. Relations (\ref{conj1}), (\ref{conj13})--(\ref{conj15}) show that these form a simple YD submodule (for fixed $k$), denoted by $V_k$. All reflections $s_i^k$ have order $d>1$, where $d$ is the order of $k$ modulo $m$, a divisor of $m/p$. For fixed $d$, there are $\varphi(d)$ such $k$, each giving a direct summand.

Second, we are given reflections of the form $\theta^k(ij)$. It is clear from relations (\ref{conj2})--(\ref{conj12}) that these determine a YD submodule, denoted by $V_0$. Hence, $Y_G$ decomposes as a direct sum
\[
Y_G=V_0\oplus \bigoplus_{i=1}^{m/p-1}V_{ip}.
\]
We will treat the cases $n=2$ and $n\geq 3$ separately. If $n=2$, the module $V_0$ is only simple if $p$ is odd. To prove this assertion, we consider relations (\ref{conj2})--(\ref{conj5}) and note that by (\ref{conj5}), $\theta^l(12)\sim \theta^k(12)$ if and only if $l$, $k$ have the same parity (using conjugation by $\theta(12)$ if $m>1$, otherwise the assertion is clear). But by (\ref{conj3}), two reflections $\theta^k(12)$, $\theta^l(12)$ are also conjugate if $k=l \mod p$. This links $k$ and $l$ of opposite parity if and only if $p$ is odd. Since in the $n=2$ case these are all relevant $\sim$ relations, we find that $V_0$ is simple if $p$ is odd, and splits as a direct sum $V^{\op{odd}}\oplus V^{\op{even}}$ of simples if $p$ is even.

It remains to show that if $n\geq 3$, we have $\theta^k(ij)\sim \theta^l(st)$ for all possible choices of $s,t,i,j,k,l$. We prove this by distinguishing several cases (assume $i,j,s,t$ pairwise distinct). For example, $\theta^k(ij)\sim \theta^l(st)$ for  $i<j<s<t$, or $i<s<j<t$ follows from 
\begin{align*}
\theta^{k-l}(is)\theta^l(st)\theta^{k-l}(is)&=\theta^k(it)\\
(jt)\theta^k(it)(jt)&=\theta^k(ij),
\end{align*}
using (\ref{conj12}) and (\ref{conj11}). The other cases to consider follow using appropriate combinations of the relations (\ref{conj5})--(\ref{conj12}).
\end{proof}

\subsection{Dimensionality}

As the YD module $Y_G$ often decomposes as a direct sum of more than one absolutely simple YD module, we can use the classification results of \cite{HV1} in the rank two case, and \cite{HV2} in higher rank cases to determine which $G=G(m,p,n)$ give infinite-dimensional Nichols algebras $\cB_G$. This will not provide a complete answer as the rank one case  --- of elementary Nichols algebras --- is not completely solved, but we will be able to derive conjectures from the conjecture made in \cite[1.1]{HLV}.

\begin{lemma}\label{indeclemma}
Let $n\geq 2$. The YD module $Y_G$ is braid-indecomposable if and only if $G\neq G(2,2,2)$.
\end{lemma}
\begin{proof}
Case 1: Let $n\geq 3$ or $p$ odd so that $V_0$ from the proof of Lemma~\ref{ranklemma} is absolutely simple. Simple YD modules are braid-indecomposable, so assume $m\neq p$. Then $Y_G$ contains reflections of the form $s_i^k\in V_k$ and $\theta^l(ij)\in V_0$. Assume that $Y_G=M'\oplus M''$, $M',M''\neq 0$. Then wlog $V_k\leq M'$, and $V_0\leq M''$. But we compute
\begin{align}
\ad(s_i^k)(\theta^l(ij))=(\ide-\Psi^2)s_i^k\otimes\theta^l(ij)&=s_i^k\otimes \theta^l(ij)+\lambda'\Psi(\theta^{l-k}(ij)\otimes s_i^k),\\
&=s_i^k\otimes \theta^l(ij)-\lambda'' s_j^k\otimes \theta^{l-k}(ij)\neq 0,
\end{align}
where $\lambda',\lambda''$ are scalars. This implies that $Y_G$ is braid-indecomposable.

Case 2: Let $n=2$ and $p=2a$ even. In this case, $V_0=V^{\op{odd}}\oplus V^{\op{even}}$. It remains to check the case where $Y_G$ has rank two, as if the rank is larger than two, there exist reflections of the form $\theta_i^k$ and the computation used in Case 1 can be applied. The rank two case applies to groups of the form $G(2a,2a,2)$. Consider $\ad(V^{\op{odd}})(V^{\op{even}})$:
\begin{align*}
(\ide-\Psi^2)\theta(ij)\otimes (ij)&=\theta(ij)\otimes (ij)-\theta^2\cdot \theta^3(ij)\otimes \theta^2(ij).
\end{align*}
This vanishes if and only if $\theta^2=1$, i.e. $m=2$. Hence $Y_G$ is braid-indecomposable unless $G=G(2,2,2)$.
\end{proof}

\begin{corollary}$~$
\begin{enumerate}
\item[(i)]
If $n\geq 3$, then all $G(m,p,n)$ where the rank of $Y_G$ is strictly larger than one give infinite-dimensional FK Nichols algebras $\cB_G$ and hence also infinite-dimensional FK algebras $\cE_G$.

In particular, this is true for all $G(m,p,n)$ where $n\geq 3$, and $m\neq p$.
\item[(ii)]
If $n=2$, then all $G(m,p,n)$ where $m\neq p$ or $p$ even have that $\cB_G$ is infinite-dimensional, with the exceptions of
\begin{align*}
D_2=S_2\times S_2=G(2,2,2), &&B_2=G(2,1,2)\cong Dih_4=I_2(4)=G(4,4,2).
\end{align*}
\end{enumerate}
\end{corollary}
\begin{proof}
This is an application of the classification results of \cite{HV1,HV2}.

Proof of (i), $\op{rank}Y_G\geq 3$: This uses \cite[Theorem 2.5]{HV2}, which applies by braid-indecomposability proved in Lemma~\ref{indeclemma} and the fact that complex reflection groups are generated by their set of reflections, which are the support of $Y_G$. Note that in this classification the maximal dimension of a simple YD summand is three, and rank three only appears together with two one-dimensional summands (type $\beta_3'$), or one one-dimensional and one two-dimensional summand (type $\beta_3''$). However, for $n\geq 3$ there is always an absolutely simple summand of dimension $mn(n-1)/2$, and the other summands have dimension $n$. Hence all such $Y_G$ of rank greater or equal to three are infinite-dimensional.

Proof for $\op{rank}Y_G=2$: For the rank to be two, assuming that $n\geq 3$, we need $m/p=2$. Then there exist a YD summand of dimension $n$ and one of dimension $mn(n-1)/2$, and $\dim Y_G=n+mn(n-1)/2$. According to \cite[Table~1]{HV1} all rank two braid-indecomposable finite-dimensional Nichols algebras of rank two have dimension between 4 and 6. No such $m,n$ exist in this case.

Proof of (ii), $\op{rank}Y_G\geq 3$: If $p$ is odd, we are in this case if and only if $m/p\geq 3$. But then there is a YD summand of dimension $m\geq 3$, while the other summands have dimension two. Again using \cite[Theorem 2.5]{HV2}, we find that no such braid-indecomposable semisimple YD module has a finite-dimensional Nichols algebra. If $p=2a$, we require $m/p\geq 2$ and have two YD summands of dimension $2an(n-1)/4=a$, and at least one of dimension two. Finite-dimensionality can only occur if $a\leq 2$ by the same result. This leaves us to only consider the case $G(4,2,2)$ which either has a skeleton of type $\alpha_3$ or is infinite-dimensional. It is in fact infinite-dimensional --- see Example~\ref{exceptions}(iii).

Proof for $\op{rank}Y_G=2$: If $p$ is odd, we need $m/p=2$, and $4\leq m+2\leq 6$. This leaves the group $G(2,1,2)$. If $p=2a$, we need $m/p=1$ and $4\leq m\leq 6$. Hence it suffices to consider the groups $G(4,4,2)$ and $G(6,6,2)$. But following \cite[Appendix~B]{HV2} $Y_{G(6,6,2)}$ is not of finite type as there is no skeleton with two dimension three vertices.

To complete the proof, we will treat the unanswered cases $G(2,1,2)$, $G(4,2,2)$, and $G(4,4,2)$ in Example~\ref{exceptions} below. Finally, note that $\cB_{D_2}$ is finite-dimensional as it is isomorphic to $\cB_{S_2}\otimes \cB_{S_2}$.
\end{proof}

\begin{example}\label{exceptions}$~$
\begin{enumerate}
\item[(i)]$G=G(2,1,2)=B_2$: We can check explicitly, using the table of actions in \cite[Example 1.2]{HV1} that $\cB_{B_2}$ is isomorphic to a 64-dimensional Hopf algebra with skeleton of type $\alpha_2$ (cf. \cite[B.1]{HV2}). The simple YD modules $V_0$ and $V_{1}$ (using the notation from Lemma \ref{ranklemma}) have support given by the conjugacy classes $\lbrace g=(12),\epsilon g=\theta(12)\rbrace$ and $\lbrace h=s_1^1,\epsilon h=s_2^1\rbrace$, respectively. Here, $\epsilon=s_1^{-1}s_2^1$. The characters are given by $\rho(\epsilon)=\rho(g)=\theta=-1$ and $\sigma(\epsilon)=\sigma(h)=\theta=-1$ (using the notation from \cite{HV1}).

The Nichols algebra $\cB_{B_2}$ hence has the Hilbert polynomial (cf. \cite[1.1]{HV1})
\begin{align}
\cH^{\cB}_{B_2}(t_1,t_2)&=(1+t_1)^2(1+t_1t_2)^2(1+t_2)^2.
\end{align}
Computations with MAGMA show that the corresponding FK algebra $\cE_{B_2}$ has a different Hilbert polynomial:
\begin{align}
\cH^{\cE}_{B_2}&=1+4t+8t^2+12t^3+16t^4+20t^5+24t^6+28t^7+\ldots.
\end{align}
Hence the Hilbert polynomials differ from degree $4$ onwards, and thus $\cB_{B_2}\ncong \cE_{B_2}$. This is the first Weyl group for which the two algebras differ.\footnote{Note that it is an open question whether $\cB_{S_n}\cong \cE_{S_n}$ for $n\geq 6$.} Note that according to \cite{KM}, $\cB_{B_2}$ is in fact infinite-dimensional.
\item[(ii)]$G=I_2(4)=G(4,4,2)$: Similarly to (i) we again see that this rank two YD module gives \cite[Example 1.2]{HV1}, with the same centralizer characters $\rho$, $\sigma$. Hence $\cB_{B_2}$ and $\cB_{I_2(4)}$ are isomorphic. This is not surprising as both groups are isomorphic to the dihedral group $Dih_4$. An explicit isomorphism is given by mapping $(12)\in B_2\mapsto (12)\in I_2(4)$, $s_1^1\in B_2\mapsto \theta(12)\in I_2(4)$.
\item[(iii)]$G=G(4,2,2)$: We claim that $\cB_G$ in this case is infinite-dimensional. As it is of rank three, we can look at all possible pairs of YD summands. The summands are given by
\begin{align*}
\op{Supp}V^{\op{even}}=\lbrace (12),\theta^2(12)\rbrace, &&\op{Supp}V^{\op{odd}}=\lbrace\theta(12),\theta^3(12)\rbrace, && \op{Supp}V_2=\lbrace s_1^2,s_2^2\rbrace.
\end{align*}
The pair $V^{\op{odd}},V_2$ gives example (i), $V^{\op{even}},V^{\op{odd}}$ gives example (ii), and it can be show that $V^{\op{even}},V_2$ again gives the same YD module as in \cite[Example~1.2]{HV1}. Hence the skeleton of $\cB_{G}$ is
\begin{align}
\vcenter{\hbox{\begin{tikzpicture}
\node (1){};
\node [circle,fill, inner sep=1] (1.1)[above of=1,node distance=0.1cm]{};
\node [circle,fill, inner sep=1] (1.2)[below of=1,node distance=0.1cm]{};
\node (2)[above right of=1,node distance=1.5cm]{};
\node [circle,fill, inner sep=1] (2.1)[above of=2,node distance=0.1cm]{};
\node [circle,fill, inner sep=1] (2.2)[below of=2,node distance=0.1cm]{};
\node  (3)[below right of=2,node distance=1.5cm]{};
\node [circle,fill, inner sep=1] (3.1)[above of=3,node distance=0.1cm]{};
\node [circle,fill, inner sep=1] (3.2)[below of=3,node distance=0.1cm]{};
\draw[dashed]  (1.east) -- (2.west) node [above, text centered,midway]{};
\draw[dashed]  (2.east) -- (3.west) node [above, text centered,midway]{};
\draw[dashed]  (1.east) -- (3.west) node [above, text centered,midway]{};
\end{tikzpicture}}}
\end{align}
Such a skeleton is not of finite type (cf. \cite[Theorem 2.5]{HV2}) and hence $\cB_G$ is infinite-dimensional.
\end{enumerate}
\end{example}

\begin{conjecture}\label{conjecture2}
Up to isomorphism, the only finite-dimensional \emph{elementary} Fomin-Kirillov Nichols algebras $\cB_G$, for $G=G(m,p,n)$, are those of the symmetric groups $G=S_2,S_3,S_4,S_5$.

Hence, the only finite-dimensional FK Nichols algebras for complex reflection groups $G(m,p,n)$ are those of $S_2\cong C_2, S_3\cong Dih_3,S_4\cong D_3, S_5, B_2\cong Dih_4, C_3,C_4, D_2\cong S_2\times S_2$.
\end{conjecture}

\begin{remark}\label{conjexpl}
Conjecture~\ref{conjecture2} is derived from the above results, combined with \cite[Conjecture~1.1]{HLV}. We observe that only $G=S_2,S_3,S_4,S_5$ give elementary Nichols algebras that occur in the table \cite[Table~9.1]{HLV}. In fact, the support of $Y_G$ always generates $G$, and $A_n$ is not a complex reflection group. It remains to check that the examples associated to $\op{Aff}(5,2)$, $\op{Aff}(5,3)$, $\op{Aff}(7,3)$, and $\op{Aff}(7,5)$ occurring in the list are not of the form $\cB_G$. For this, note that the dimensions of the corresponding YD modules are 5 and 7. The only simple YD modules $Y_G$ of these dimensions belong to the dihedral groups $Dih_5=G(5,5,2)$ and $Dih_7=G(7,7,2)$. These however correspond to the affine groups $\op{Aff}(5,-1)$ and $\op{Aff}(7,-1)$, which are not isomorphic to the ones listed. In fact, MAGMA computations show that the Hilbert series are 
\begin{align}
\cH^\cE_{Dih_5}(t)&=1+5t+16t^3+45t^3+121 t^4+O(t^5), & \cH^\cE_{Dih_7}(t)&=1+7t+36t^2+175t^3+O(t^4),
\end{align}
which differ from all Hilbert series stated in \cite[Table~9.1]{HLV} in degree two already.
The groups $Dih_3=G(3,3,2)$ and $D_3=G(2,2,3)$ also give finite-dimensional FK Nichols algebras, however are isomorphic to $\cB_3$ and $\cB_4$, respectively. Hence, up to isomorphisms, $\cB_2,\cB_3,\cB_4,\cB_5$ would be the only finite-dimensional elementary FK Nichols algebras according to \cite[Conjecture~1.1]{HLV}.
\end{remark}

\begin{conjecture}\label{conjecture3}
Up to isomorphism, the only finite-dimensional Fomin--Kirillov algebras $\cE_G$ for non-exceptional indecomposable Coxeter groups are those for $G=S_2,S_3,S_4,S_5$.
\end{conjecture}

\begin{remark}\label{concluding}
The only cases left to check in Conjectures \ref{conjecture2} and \ref{conjecture3} are among the groups $G(m,m,n)$, where $n\geq 3$ or $m$ odd. Recall that the algebra $\cB_{B_2}$ is finite-dimensional, but $\cE_{B_2}$ is not.
\end{remark}

%%%%%%%%%%%%%%%%%%%%%%%%%%%%%%%%%%%%%%%%%%%%%%%%%%%%%%%%%%%%%%%%%%%%%%%%%%%%%%%%%%%%%%%%%%%%%%%%%%%%%%%%%%%%%%%%%%%%

\section*{Acknowledgements}
The author likes to thank Yuri Bazlov for introducing him to Fomin--Kirillov algebras. Some of this work was conducted during a visit of the author at University of Marburg, and the author is grateful for the hospitality offered by the Department of Mathematics. Special thanks go to thank Istv\'an Heckenberger for teaching the author about Weyl groupoids and providing many helpful insights and guidance which made this work possible. The visit to Marburg was supported by the DFG Schwerpunkt ``Darstellungstheorie". The author is further supported by an EPSRC Doctoral Prize at UEA.

%\addcontentsline{toc}{section}{References}
%\nocite{*}
\bibliography{biblio}

\begin{bibdiv}
\begin{biblist}

\bib{AFGV}{article}{
      author={Andruskiewitsch, N.},
      author={Fantino, F.},
      author={Gra{\~n}a, M.},
      author={Vendramin, L.},
       title={Finite-dimensional pointed {H}opf algebras with alternating
  groups are trivial},
        date={2011},
        ISSN={0373-3114},
     journal={Ann. Mat. Pura Appl. (4)},
      volume={190},
      number={2},
       pages={225\ndash 245},
         url={http://dx.doi.org/10.1007/s10231-010-0147-0},
      review={\MR{2786171 (2012c:16095)}},
}

\bib{AS3}{article}{
      author={Andruskiewitsch, Nicol{\'a}s},
      author={Schneider, Hans-J{\"u}rgen},
       title={Finite quantum groups and {C}artan matrices},
        date={2000},
        ISSN={0001-8708},
     journal={Adv. Math.},
      volume={154},
      number={1},
       pages={1\ndash 45},
         url={http://dx.doi.org/10.1006/aima.1999.1880},
      review={\MR{1780094 (2001g:16070)}},
}

\bib{AS}{incollection}{
      author={Andruskiewitsch, Nicol{\'a}s},
      author={Schneider, Hans-J{\"u}rgen},
       title={Pointed {H}opf algebras},
        date={2002},
   booktitle={New directions in {H}opf algebras},
      series={Math. Sci. Res. Inst. Publ.},
      volume={43},
   publisher={Cambridge Univ. Press, Cambridge},
       pages={1\ndash 68},
         url={http://dx.doi.org/10.2977/prims/1199403805},
      review={\MR{1913436 (2003e:16043)}},
}

\bib{AS2}{article}{
      author={Andruskiewitsch, Nicol{\'a}s},
      author={Schneider, Hans-J{\"u}rgen},
       title={On the classification of finite-dimensional pointed {H}opf
  algebras},
        date={2010},
        ISSN={0003-486X},
     journal={Ann. of Math. (2)},
      volume={171},
      number={1},
       pages={375\ndash 417},
         url={http://dx.doi.org/10.4007/annals.2010.171.375},
      review={\MR{2630042 (2011j:16058)}},
}

\bib{Baz}{article}{
      author={Bazlov, Yuri},
       title={Nichols-{W}oronowicz algebra model for {S}chubert calculus on
  {C}oxeter groups},
        date={2006},
        ISSN={0021-8693},
     journal={J. Algebra},
      volume={297},
      number={2},
       pages={372\ndash 399},
         url={http://dx.doi.org/10.1016/j.jalgebra.2006.01.037},
      review={\MR{2209265 (2007b:20083)}},
}

\bib{BB}{article}{
      author={Bazlov, Yuri},
      author={Berenstein, Arkady},
       title={Braided doubles and rational {C}herednik algebras},
        date={2009},
     journal={Adv. Math.},
      volume={220},
      number={5},
       pages={1466\ndash 1530},
      review={\MR{2493618 (2010e:16046)}},
}

\bib{FK}{incollection}{
      author={Fomin, Sergey},
      author={Kirillov, Anatol~N.},
       title={Quadratic algebras, {D}unkl elements, and {S}chubert calculus},
        date={1999},
   booktitle={Advances in geometry},
      series={Progr. Math.},
      volume={172},
   publisher={Birkh\"auser Boston, Boston, MA},
       pages={147\ndash 182},
      review={\MR{1667680 (2001a:05152)}},
}

\bib{FP}{article}{
      author={Fomin, Sergey},
      author={Procesi, Claudio},
       title={Fibered quadratic {H}opf algebras related to {S}chubert
  calculus},
        date={2000},
        ISSN={0021-8693},
     journal={J. Algebra},
      volume={230},
      number={1},
       pages={174\ndash 183},
         url={http://dx.doi.org/10.1006/jabr.1999.7957},
      review={\MR{1774762 (2001j:16058)}},
}

\bib{SARNA}{misc}{
      author={Gra\~na, M.},
      author={Heckenberger, I.},
      author={Vendramin, L.},
       title={Sarna, a software system for arithmetic root systems and nichols
  algebras},
         how={Available at \url{http://code.google.com/p/sarna-python/}},
        date={2011},
}

\bib{Hec1}{article}{
      author={Heckenberger, I.},
       title={The {W}eyl groupoid of a {N}ichols algebra of diagonal type},
        date={2006},
        ISSN={0020-9910},
     journal={Invent. Math.},
      volume={164},
      number={1},
       pages={175\ndash 188},
         url={http://dx.doi.org/10.1007/s00222-005-0474-8},
      review={\MR{2207786 (2007e:16047)}},
}

\bib{Hec4}{inproceedings}{
      author={Heckenberger, I.},
       title={Classification of arithmetic root systems of rank 3},
        date={2007},
   booktitle={Proceedings of the {XVI}th {L}atin {A}merican {A}lgebra
  {C}olloquium ({S}panish)},
      series={Bibl. Rev. Mat. Iberoamericana},
   publisher={Rev. Mat. Iberoamericana, Madrid},
       pages={227\ndash 252},
      review={\MR{2500361 (2010j:17021)}},
}

\bib{Hec3}{article}{
      author={Heckenberger, I.},
       title={Rank 2 {N}ichols algebras with finite arithmetic root system},
        date={2008},
        ISSN={1386-923X},
     journal={Algebr. Represent. Theory},
      volume={11},
      number={2},
       pages={115\ndash 132},
         url={http://dx.doi.org/10.1007/s10468-007-9060-7},
      review={\MR{2379892 (2009aMR2278041:16080)}},
}

\bib{Hec2}{article}{
      author={Heckenberger, I.},
       title={Classification of arithmetic root systems},
        date={2009},
        ISSN={0001-8708},
     journal={Adv. Math.},
      volume={220},
      number={1},
       pages={59\ndash 124},
         url={http://dx.doi.org/10.1016/j.aim.2008.08.005},
      review={\MR{2462836 (2009k:17017)}},
}

\bib{HLV}{article}{
      author={Heckenberger, I.},
      author={Lochmann, A.},
      author={Vendramin, L.},
       title={Nichols algebras with many cubic relations},
        date={2015},
        ISSN={0002-9947},
     journal={Trans. Amer. Math. Soc.},
      volume={367},
      number={9},
       pages={6315\ndash 6356},
         url={http://dx.doi.org/10.1090/S0002-9947-2015-06231-X},
      review={\MR{3356939}},
}

\bib{HV1}{article}{
      author={Heckenberger, I.},
      author={Vendramin, L.},
       title={{The classification of Nichols algebras over groups with finite
  root system of rank two}},
        date={2013-11},
     journal={ArXiv e-prints},
      eprint={1311.2881},
}

\bib{HV2}{article}{
      author={Heckenberger, I.},
      author={Vendramin, L.},
       title={{A classification of Nichols algebras of semi-simple
  Yetter-Drinfeld modules over non-abelian groups}},
        date={2014-12},
     journal={ArXiv e-prints},
      eprint={1412.0857},
}

\bib{Kha}{article}{
      author={Kharchenko, V.~K.},
       title={Connected braided {H}opf algebras},
        date={2007},
        ISSN={0021-8693},
     journal={J. Algebra},
      volume={307},
      number={1},
       pages={24\ndash 48},
         url={http://dx.doi.org/10.1016/j.jalgebra.2006.02.037},
      review={\MR{2278041 (2007i:16063)}},
}

\bib{KM}{article}{
      author={Kirillov, Anatol~N.},
      author={Maeno, Toshiaki},
       title={Noncommutative algebras related with {S}chubert calculus on
  {C}oxeter groups},
        date={2004},
        ISSN={0195-6698},
     journal={European J. Combin.},
      volume={25},
      number={8},
       pages={1301\ndash 1325},
         url={http://dx.doi.org/10.1016/j.ejc.2003.11.006},
      review={\MR{2095483}},
}

\bib{Maj2}{incollection}{
      author={Majid, Shahn},
       title={Noncommutative differentials and {Y}ang-{M}ills on permutation
  groups {$S\sb n$}},
        date={2005},
   booktitle={Hopf algebras in noncommutative geometry and physics},
      series={Lecture Notes in Pure and Appl. Math.},
      volume={239},
   publisher={Dekker, New York},
       pages={189\ndash 213},
      review={\MR{2106930 (2006b:58007)}},
}

\bib{MS}{incollection}{
      author={Milinski, Alexander},
      author={Schneider, Hans-J{\"u}rgen},
       title={Pointed indecomposable {H}opf algebras over {C}oxeter groups},
        date={2000},
   booktitle={New trends in {H}opf algebra theory ({L}a {F}alda, 1999)},
      series={Contemp. Math.},
      volume={267},
   publisher={Amer. Math. Soc., Providence, RI},
       pages={215\ndash 236},
         url={http://dx.doi.org/10.1090/conm/267/04272},
      review={\MR{1800714 (2001k:16074)}},
}

\bib{ST}{article}{
      author={Shephard, G.~C.},
      author={Todd, J.~A.},
       title={Finite unitary reflection groups},
        date={1954},
        ISSN={0008-414X},
     journal={Canadian J. Math.},
      volume={6},
       pages={274\ndash 304},
      review={\MR{0059914 (15,600b)}},
}

\bib{WZZ}{article}{
      author={Wu, Weicai},
      author={Zhang, Shouchuan},
      author={Zhang, Yao-Zhong},
       title={Finite dimensional {N}ichols algebras over finite cyclic groups},
        date={2014},
        ISSN={0949-5932},
     journal={J. Lie Theory},
      volume={24},
      number={2},
       pages={351\ndash 372},
      review={\MR{3235894}},
}

\end{biblist}
\end{bibdiv}
\bibliographystyle{amsrefs}%agsm or dcu

\end{document}